\documentclass[12pt]{article}
\usepackage{amssymb}

\usepackage{amsmath}
\usepackage{amsthm}
\usepackage{tikz}
\usepackage{amscd}
\usepackage{latexsym}
\usepackage{mathrsfs}
\usepackage{color}
\usepackage{graphicx}
\usepackage[noadjust]{cite}
\usepackage[all,pdf]{xy}
\usepackage{enumerate}
\usepackage{doi}
\usepackage{enumitem}

\newtheorem{thm}{Theorem}[section]

\newtheorem{prop}[thm]{Proposition}
\newtheorem{lem}[thm]{Lemma}
\newtheorem{cor}[thm]{Corollary}

\numberwithin{equation}{section}

\begin{document}

\title{{\bf Nilpotent groups, solvable groups and factorizable inverse monoids}
\footnote{This paper is supported by National Natural Science Foundation of China (11971383, 12201495, 12371024).}}
\author{{\bf Dong-lin Lei} $^a$, {\bf Jin-xing Zhao} $^b$, {\bf Xian-zhong Zhao} $^c$ 
  \footnote{Corresponding author. E-mail: zhjxing@imu.edu.cn} \\
   {\small $^a$ School of Mathematics, Northwest University}\\
   {\small Xi'an, Shaanxi, 710127, P. R. China}\\
   {\small $^b$ School of Mathematical Sciences, Inner Mongolia University}\\
   {\small Hohhot, Inner Mongolia, 010021, P. R. China}\\
    {\small $^c$ School of Mathematics and Data Science},\\
   {\small Shaanxi University of Science and Technology} \\
   {\small Xian, Shaanxi, 710021, P.R. China}\\}
\date{}
\maketitle
\vskip -4pt
\baselineskip 16pt

{\small
\noindent\textbf{ABSTRACT}
In this paper subcentral (resp., central) idempotent series and composition subcentral (resp., central) idempotent series in an inverse semigroup are introduced and investigated. It is shown that if $S=EG$ is a factorizable inverse  monoids with semilattice $E$ of idempotents and the group $G$ of units such that the natural connection $\theta$ is a dual isomorphism from $E$ to a sublattice of $L(G)$, then any two composition subcentral (resp., central) idempotent series in $S$ are isomorphic. It may be considered as an appropriate analogue in semigroup theory of Jordan-H\"{o}lder Theorem in group theory. Based on this,
$G$-nilpotent and $G$-solvable inverse monoids are also introduced and studies. Some characterizations of the coset monoid of nilpotent groups and solvable groups are given.  This extends the main result in Semigroup Forum 20: 255-267, 1980 and also provides another effective approach for the study of nilpotent and solvable groups. Finally, some open problems related to nilpotent and solvable groups are translated to semigroup theory, which may be helpful for us to solve these open problems.

 }

\section{Introduction}

Throughout this paper, we shall confine ourselves to the inverse monoids with zero.
Recall that an inverse monoid $S$ is said to be \emph{factorizable} if every element lies beneath an
element of the group of units with respect to the natural partial order. That is to say, $S=GE$ (or equivalently, $S=EG$), where $E$ is the semilattice of idempotents of $S$ and $G$ is the group of units of $S$ (see \cite{lawson}). Factorizable inverse monoids and their generalizations play an important role in the theory of inverse semigroups and have been studied extensively in the literature, see for example \cite{chen, easdown, james, james4, Fitz, lawson1, mcali2}.

Let $G$ be a group and ${\cal K}(G)$ denote the set of all right cosets of all subgroups of $G$. Then ${\cal K}(G)$ equipped with the binary operation
$$Ha* Kb = \langle H, K^{a^{-1}}\rangle ab$$
becomes an inverse monoid with zero (see \cite{schei} or \cite{mcali1}), which is called the coset monoid of group $G$. Every coset monoid is a factorizable inverse monoid. Coset monoids have attracted the attention of many scholars such as Schein, McAlister, East and Ara\'{u}jo et al.. There is a series of papers in the literature devote to study them (see \cite{lawson, schei, schei2, mcali1, james2, james3, shum, arau, petrich}).
In particular it was shown in \cite{mcali1} that an inverse semigroup $S$ is isomorphic to the coset monoid ${\cal K}(G)$ for some group $G$ if and only if $S$ is factorizable and the natural connection $\theta$ of the semilattice $E$ idempotents of $S$ into the lattice $L(G)$ of subgroups of $G$, defined by $e\theta = \{h \in H: eh = e\}$ for each idempotent $e \in S$, is a dual isomorphism.
Along this line of investigation, we will go on to the study of the factorizable inverse monoids and coset monoids in this paper.

Firstly, we shall introduce and investigate subcentral (resp., central) idempotent series and composition subcentral (resp., central) idempotent series in an inverse semigroup. It is shown that any two composition subcentral (resp., central) idempotent series of $S$ are isomorphic for factorizable inverse  monoids $S=EG$ such that the natural connection $\theta$ is a dual isomorphism from $E$ to a sublattice of $L(G)$, where $E$ and $G$ denote the semilattice of idempotents of $S$ and the group of units of $S$ respectively. This provides an appropriate analogue in semigroup theory of Jordan-H\"{o}lder Theorem in group theory in section 2. Also, $G$-nilpotent and $G$-solvable inverse monoids are introduced in section 3. Based on this, some characterizations of the coset monoid of nilpotent and solvable groups are given. In section 4 some open problems related to nilpotent and solvable groups are translated to semigroup theory, which may be helpful for us to solve these open problems.

For the notation and terminology used in this paper, the reader is referred to Howie \cite{howie} and Lawson \cite{lawson} for backgrounds on semigroups and inverse semigroups, to Robinson \cite{robinson} for knowledge of group theory.

\section{The analogue of Jordan-H\"{o}lder Theorem}

It is well-known that Jordan-H\"{o}lder Theorem plays an fundamental role in group theory. In this section we shall provide an appropriate analogue in inverse semigroup theory of Jordan-H\"{o}lder Theorem in group theory.
For this, the following notation and terminology are needed.
Recall that an element in an inverse semigroup $S$ is called central if $cs=sc$ for every $s$ in $S$.
The set $Z(S)$ of all central elements, called the centre of $S$, which forms an inverse subsemigroup of $S$.
Also, $eS$ forms an inverse submonoid  with identity $e$ for every central idempotent $e$ in $S$.
It is well known that $H_e$ (the ${\cal H}$-class of $S$ containing $e$) is just the group of units of $eS$.

Let $S$ be an inverse semigroup. It is well known that $S$ is a partially ordered inverse semigroup under the natural order (write it as $\leq$).
As usual, $a^{\upharpoonright}$ denotes the filter generated by an element $a$ in $S$, which is of course an inverse subsemigroup of $S$ when $a$ is idempotent in $S$.
For idempotents $e$ and $f$ in $S$ with $f$ central in $e^\upharpoonright$, $fe^\upharpoonright$ forms an inverse submonoid of $e^\upharpoonright$, denoted by $F_{e,f}$. Also, we shall write $UF_{e,f}$ to denote the group of units of $F_{e,f}$.
A chain
$$0=e_0\leq e_1\leq \cdots\leq e_{n-1} \leq e_n = e  \eqno(*)$$
of idempotents in $S$ is said to be a \emph{subcentral (resp., central) idempotents series between 0 and $e$ of length $n$} if $e_{i+1}$ is central in $e_{i}^{\upharpoonright}$ (resp., $S$) for every $0\leq i\leq n-1$. The length of the shortest subcentral idempotent series between 0 and $e$ is called the \emph{defect} of $e$.
Each $e_i$ in the subcentral idempotents series $(*)$ is said to be \emph{subcentral} in $S$.
Also, a subcentral (resp., central) idempotents series between 0 and 1 is said to be a \emph{subcentral (resp., central) idempotents series of $S$}. It is easy to see that an idempotent $e$ is central, if and only if, $0\leq e\leq 1$ is a subcentral idempotent series of $S$, if and only if, $e$ is subcentral with defect 0 or 1. Thus every central idempotent is subcentral in $S$. Further, for idempotents $e,f$ in $S$ with $f$ central in $e^\upharpoonright$, if $e$ is subcentral and $(*)$ is a subcentral idempotent series between 0 and $e$, then
$$0=e_0\leq e_1\leq \cdots\leq e_{n-1} \leq e_n = e\leq f$$
becomes a subcentral idempotent series between 0 and $f$ and so $f$ is subcentral in $S$.
This shows that a central idempotent in the filter generated by a subcentral idempotent in $S$ is subcentral in $S$.

It is clear that the subcentral idempotent series $(*)$ between 0 and $e$ determines a descending chain
$$S=e_0^{\upharpoonright}\supseteq e_1^{\upharpoonright}\supseteq \cdots \supseteq e_{n-1}^{\upharpoonright}\supseteq e_n^{\upharpoonright}$$
of inverse subsemigroups of $S$.
Further, the mapping $\varphi_{e_i,e_{i+1}}: e_i^{\upharpoonright}\rightarrow e_{i+1}^{\upharpoonright}$ defined by $(x)\varphi_{e_i,e_{i+1}}=e_{i+1}x$ is a semigroup homomorphism.
Each homomorphic image $(e_i^\upharpoonright)\varphi_{e_i,e_{i+1}}$ ($0\leq i\leq n-1$), which is equal to $F_{e_i,e_{i+1}}$ (see Lemma \ref{phie}), is called the \emph{factors} of the series $(*)$.
Two subcentral (resp., central) idempotent series of an inverse semigroup are said to be \emph{isomorphic} if they have the same length and there is a one-to-one correspondence between their factors such that corresponding factors are isomorphic.
A subcentral (resp., central) idempotent series of an inverse semigroup is called a \emph{subcentral refinement} (resp., \emph{central refinement}) of another subcentral (resp., central) idempotent series of the inverse semigroup if the terms of the second series all occurs as terms of the first.
A subcentral (resp., central) idempotent series of an inverse semigroup is said to be a \emph{composition} subcentral (resp., central) idempotent series if it has no repeated terms and can not be refined except by repeating terms.

Let $S$ be a factorizable inverse monoid with semilattice of idempotents $E$ and group of units $G$, the mapping $\theta$ from $E$ into the lattice $L(G)$ of subgroups of $G$ defined by
$$e\theta = \{g \in G: eg = e\}=e^\upharpoonright \cap G$$
is anti-isotone, which is called the \emph{natural connection} between $E$ and $L(G)$ (see \cite{mcali1}).
We have

\begin{lem}\label{fgeq}
Let $S$ be a factorizable inverse monoid with semilattice of idempotents $E$ and group of units $G$. Then the following are true:
\begin{itemize}
\item[${ (i)}$] $e^\upharpoonright = [e,1]\cdot e\theta $;

\item[${(ii)}$] $eS = [0,e]\cdot G$;

\item[${(iii)}$] $F_{e,f}=f e^\upharpoonright =[e,f]\cdot e\theta$ and $UF_{e,f}=f\cdot e\theta$,

\end{itemize}

\noindent where  $e,\,f$ are two idempotents in $S$  such that $f$ is central in $e^\upharpoonright$, $[e,f]$  denote the interval $\{e' \mid e \leq e' \leq f\}$ in the semilattice $E$.
\end{lem}

\begin{proof}
(i) It is clear that $[e,1]\cdot e\theta=\{e'g \mid e\leq e'\in E, e\leq g\in G\}$.

Suppose that $s=e'g$, where $ e\leq e'\in E, e\leq g\in G$.
Then $s\in e^\upharpoonright$.
Conversely, if $s=e'g\in e^\upharpoonright$, then  $g\geq ee'g=es=e$, since $ee'\in E$ is idempotent. Also, we can see from \cite[Lemma 2.1(iv)]{chen} that $ee'=e$, i.e., $e'\geq e$. This shows that $e^\upharpoonright = [e,1]\cdot e\theta$, as required.

(ii) It is clear that $[0,e]\cdot G=\{e'g \mid e'\leq e, g\in G \}$.

Suppose that $s\in eS$. Then $s=ee'g$ for some $e'\in E$ and $g\in G$,  since $S$ is factorizable.
Thus $s\in \{e'g \mid  e'\leq e, g\in G \}$. Conversely, suppose that
$$t=f'g \in \{e'g \mid e'\leq e, g\in G  \},$$
where $f'\in E$, $f'\leq e$ and $g\in G$. Then we have $f'=ef'$ and so $t=f'g=ef'g\in eS$. This shows that $eS=[0,e]\cdot G$, as required.

(iii) It is clear that $[e,f]\cdot e\theta=\{e'g \mid e\leq e'\leq f, e\leq g\in G\}$.

We can see from (i) and (ii) that $F_{e,f}=\{e'g \mid e\leq e'\leq f, e\leq g\in G\}=[e,f]\cdot e\theta$. To show the remaining half, i.e., $UF_{e,f}=f\cdot e\theta$, we first prove that $f$ is the identity of $F_{e,f}$. For every given $s\in F_{e,f}$, suppose that $s=e'g$,
where $e\leq e'\leq f$ and $e\leq g\in G$, since $F_{e,f}=\{e'g \mid e\leq e'\leq f, e\leq g\in G\}$. Notice that $f$ is central in $e^\upharpoonright$, we have that $gfg^{-1}=gg^{-1}f=f$. This implies that $sf=e'gf=e'gfg^{-1}g=e'fg=e'g=s$ and $fs=fe'g=e'g=s$ and so $f$ is the identity of $F_{e,f}$. On the other hand, we have
$$\begin{array}{llllll}
s = e'g\in UF_{e,f}&\iff  ss^{-1}=s^{-1}s=f&\iff (e'g)(e'g)^{-1}=(e'g)^{-1}(e'g)=f  \\
&\iff  e'=g^{-1}e'g=f&\iff s=fg.
\end{array}$$
This shows that $UF_{e,f}=\{fg \mid e\leq g\in G \}=f\cdot e\theta$. This completes our proof.
\end{proof}

The above lemma tells us that in a factorizable inverse monoid $S$, the filter $e^\upharpoonright$ generated by an idempotent $e$ is also a factorizable inverse monoid with the semilattice $[e,1]$ of idempotents and the group $e\theta$ of units.
It is also  easy to verify that if $f$ is a central idempotent in $e^\upharpoonright$, then $fe^\upharpoonright$ forms an inverse subsemigroup of $e^\upharpoonright$ with  the semilattice $[e,f]$ of idempotents and the group $f\cdot e\theta$ of units.

In the remainder of paper we shall write $H\leq_g G$ (resp., $H\unlhd G$) to signify always that $H$ is a subgroup (resp., normal subgroup) of a group $G$.
In the following we shall give some preliminary results which will be useful in the sequel.

\begin{lem}\label{thecon}
Let $S$ be an inverse monoid with semilattice of idempotents $E$ and group of units $G$. Then
\begin{itemize}
\item[${ (i)}$] $(g^{-1}eg)\theta=g^{-1}(e\theta)g$ for every $e\in E$ and $g\in G$;

\item[${(ii)}$] An idempotent $e$ is central in $S$ if and only if $e\theta\unlhd G$, when $S$ is factorizable and $\theta$ is injective;

\item[${(iii)}$] The idempotent $e\vee f$ is central in $f^\upharpoonright$ (resp., central in $S$), when $S$ is factorizable and $e, \, f\in E$ such that $e$ is central in $S$ (resp., $e, \, f$ are central in $S$) and $e\vee f$ exists;

\item[${(iv)}$] $e_1\vee f$ is central in $(e_2\vee f)^\upharpoonright$, when $S$ is factorizable and $e_1, \, e_2, \, f\in E$ such that $e_1$ is central in $e_2^\upharpoonright$ and $e_1\vee f$, $e_2\vee f$ exist;

\item[${(v)}$] $e_1f$ is central in $(e_2f)^\upharpoonright$, when $S$ is factorizable, $\theta$ is a dual isomorphism from $E$ to a sublattice of $L(G)$ and $e_1, \, e_2, \, f\in E$ such that $e_1$ is central in $e_2^\upharpoonright$ and $f$ is central in $S$.
\end{itemize}
\end{lem}

\begin{proof}
(i) For every $x\in G$, we have
$$ x\in g^{-1}(e\theta)g \Longleftrightarrow gx g^{-1}\in e\theta \Longleftrightarrow gx g^{-1}\geq e \Longleftrightarrow x\geq g^{-1}eg \Longleftrightarrow x\in(g^{-1}eg)\theta. $$
This shows that $(g^{-1}eg)\theta=g^{-1}(e\theta)g$.

(ii) It is easy to  see that $e\theta\unlhd G$ if $e$ is central in $S$. Conversely, suppose that $e\theta\unlhd G$. Then, we can see from (i) that
$$(g^{-1}eg)\theta=g^{-1}(e\theta)g=e\theta$$
for every $g\in G$. It follows that $g^{-1}eg=e$, i.e., $eg=ge$ for every $g\in G$, since $\theta$ is injective. Now, we have that $es=efg=feg=fge=se$ for every given $s=fg\in S$, where $f\in E$ and $g\in G$. This shows that $e$ is central in $S$, as required.

(iii) \cite[Section 1.4 Lemma 15]{lawson} tells us that $e\vee f$ is an idempotent, we only need to show that $e\vee f$ is central in $f^\upharpoonright$. For $s=hg\in f^\upharpoonright$, where $h\in E$, $g\in G$, we may assume that $g\geq f$ by Lemma \ref{fgeq} and so $gf=f=fg$. Also, $gg^{-1}=g^{-1}g=1\geq e, \, f$. It follows from \cite[Section 1.4 Proposition 18]{lawson} that
$$(e\vee f)s=(e\vee f)hg=h(e\vee f)g=h(eg\vee fg)=h(ge\vee gf)=hg(e\vee f)=s(e\vee f).$$
This shows that $e\vee f$ is central in $f^\upharpoonright$.

The prove of the case that $e, \, f$ are central in $S$ is similar.

(iv) \cite[Section 1.4 Lemma 15]{lawson} tells us that $e_i\vee f$, $i=1,2$, are idempotents. Also, $e_1\geq e_2$  tells us that $e_1 \vee f\geq e_2 \vee f$. For every $s=hg\in(e_2\vee f)^\upharpoonright$, where $h\in E$, $g\in G$, we may assume that $g\geq e_2\vee f\geq e_2, f$ by Lemma \ref{fgeq} and so $gf=f=fg$. Notice that $gg^{-1}=g^{-1}g=1\geq e_1, f$, we can see from \cite[Section 1.4 Proposition 18]{lawson} that
\begin{align*}
(e_1\vee f)s&=(e_1\vee f)hg=h(e_1\vee f)g=h(e_1g\vee fg) \\
&=h(ge_1\vee gf)=hg(e_1\vee f)=s(e_1\vee f).
\end{align*}
This shows that $e_1\vee f$ is central in $(e_2\vee f)^\upharpoonright$.

(v) It is clear that $e_1f\geq e_2f$, since $e_1\geq e_2$. Also, for every $s=hg\in (e_2f)^\upharpoonright$, where $h\in E$ and $g\in G$, we may assume that $h\geq e_2f$ and $g\geq e_2f$ by Lemma \ref{fgeq}, that is, $g\in (e_2f)\theta$. Notice that $f$ is central in $S$ and $\theta$ is a dual isomorphism from $E$ to a sublattice of $L(G)$, it follows that $f\theta\unlhd G$ and
$$ (e_2f)\theta=(e_2\wedge f)\theta=(e_2\theta)\vee(f\theta)=(e_2\theta)(f\theta).$$
This implies that there exist $g_1\in e_2\theta$ and $g_2\in f\theta$ such that $g=g_1g_2$ and so
\begin{align*}
s(e_1f)&=hg_1g_2fe_1= hg_1fe_1=hg_1e_1f=he_1g_1f  \\
&=he_1g_1g_2f=he_1fg_1g_2=e_1fhg_1g_2=(e_1f)s.
\end{align*}
This shows that $e_1f$ is central in $(e_2f)^\upharpoonright$, as required.
\end{proof}

As a consequence of Lemma \ref{thecon} (iii) we have

\begin{cor}\label{subcs}
Let $S$ be a factorizable inverse monoid with semilattice of idempotents $E$.
\begin{itemize}
\item[${(i)}$] If $E$ is also  a upper semilattice (i.e., a lattice) under natural partial order, $e, \, f\in E$ and $e$ is subcentral in $S$, then $e\vee f$ is a subcentral idempotent in $f^\upharpoonright$.
\item[${(ii)}$] If $G$ is the group of units and the natural connection $\theta$ is a dual isomorphism from $E$ to a sublattice of $L(G)$,
then every chain $e_1\leq e_2\leq \ldots\leq e_n$ of subcentral idempotents in $S$ can be refined to a subcentral idempotent series of $S$.
\end{itemize}
\end{cor}

\begin{proof} (i) Omitted.

(ii) For each $e_i$ ($1 \leq i \leq n$), we can get a subcentral idempotent series
$$ 0=e_{i0}\leq e_{i1}\leq\cdots\leq e_{in_i}=e_i, \,\, i=1,2,\ldots,n,$$
between 0 and $e_i$,  since $e_1,e_2,\ldots,e_n$ are all subcentral in $S$. Write $ e_{i,n_i+j}=e_i\vee e_{i+1,j}$,
for any $0\leq i\leq n$ and $1\leq j\leq n_{i+1}$. Then we have
$$e_{in_i}=e_i=e_i\vee 0=e_i\vee e_{i+1,0}\leq e_i\vee e_{i+1,1}=e_{i,n_i+1}, $$
$$e_{i,n_i+j}=e_i\vee e_{i+1,j}\leq e_i\vee e_{i+1,j+1}=e_{i,n_i+j+1}, $$
$$e_{i,n_i + n_{i+1}} = e_i \vee e_{i+1,n_{i+1}} = e_i\vee e_{i+1}=e_{i+1}$$
for every $0\leq i\leq n$ and $1\leq j\leq n_{i+1}$. Also, we can see from Lemma \ref{thecon} (iii) that
$ e_{i,n_i+j+1}$ is central in $e_{i,n_i+j}^\upharpoonright$. This shows that
$$0=e_{10}\leq e_{11}\leq\cdots\leq e_{1n_1}=e_1
= e_{1,n_1+1}\leq\cdots\leq e_{1,n_1+n_2}=e_2
= e_{2,n_2+1}\leq\cdots\leq e_n$$
is a subcentral idempotent series of $S$.
\end{proof}

The following is easy to be verified and so its proof is omitted.

\begin{lem}\label{phie}
Let $e$, $f$ be idempotents in an inverse semigroup $S$ such that $f$ is central in $e^\upharpoonright$. Then the mapping $\varphi_{e,f}: e^\upharpoonright\rightarrow f^\upharpoonright$ defined by $a\varphi_{e,f}=fa$ is a homomorphism, ${\rm im}\varphi_{e,f}=fe^\upharpoonright=F_{e,f} $ and $\ker\varphi_{e,f}=\{(a,b)\in S\times S \mid fa=fb\}$.
\end{lem}

In the following we shall write $\ker\varphi_{e,f}$ in the above lemma simply as $\rho_{e,f}$.
Also, $\varphi_{0,e}$ (resp., $\rho_{0,e}$) will be simply write as $\varphi_e$ (resp., $\rho_e$) if $e$ is a central idempotent in $S$.

\begin{lem}\label{iso2}
Let $T$ be a subsemigroup of a semigroup $S$ and $\rho$ be a congruence on $S$. Write $U = \{s \in S \mid (\exists\, t \in T) \,\, s \rho t\}$. Then
\begin{itemize}
\item[${ (i)}$] $U$ is a subsemigroup of $S$, and it is also  an inverse semigroup if $S$ is an inverse semigroup;

\item[${ (ii)}$] $\rho_{_U}=\rho\cap (U\times U)$ and $\rho_{_T}=\rho\cap (T\times T)$ are congruences on $U$ and $T$ resp., and $U/\rho_{_U}\cong T/\rho_{_T}$.
\end{itemize}
\end{lem}

\begin{proof}
(i) For every $s,t\in U$, there exist $s_1, t_1\in T$ such that $(s_1,s),(t_1,t)\in\rho$. Notice that $T$ is a subsemigroup of $S$, we have that $(s_1t_1, st)\in \rho$ and so $st\in U$. This shows that $U$ is a subsemigroup of $S$. Also, we have that $(s_1^{-1},s^{-1})\in\rho$ if $S$ is an inverse semigroup. Hence $s^{-1}\in U$ and so $U$ is an inverse subsemigroup of $S$.

(ii) It is easy to see that $\rho_{_U}=\rho\cap (U\times U)$ and $\rho_{_T}=\rho\cap (T\times T)$ are congruences on $U$ and $T$ respectively. To show the remaindering half, we shall consider the mapping $\psi: T\rightarrow U/\rho_{_U}$  defined by $t\psi=t/\rho_{_U}$. We have
$$ (s\psi)(t\psi)=(s/\rho_{_U})(t/\rho_{_U})=st/\rho_{_U}=(st)\psi$$
for every $s,t\in T$. This shows that $\psi$ is a semigroup homomorphism from $T$ to $U/\rho_{_U}$. We also have that $\ker\psi=\rho_{_T}$, since
$$ s\psi=t\psi \Longleftrightarrow s/\rho_{_U}=t/\rho_{_U} \Longleftrightarrow (s,t)\in\rho \Longleftrightarrow s/\rho_{_T}=t/\rho_{_T}.$$
Notice that there exist $s_1\in T$ such that $(s_1,s)\in\rho$ for every given $s\in U$, it follows that $s_1\psi=s_1/\rho_{_U}=s/\rho_{_U}$ and so $\psi$ is surjective. This implies that $U/\rho_{_U}\cong T/\rho_{_T}$, as required.
\end{proof}

\begin{lem}\label{iso2pe}
Let $S$ be a factorizable inverse  monoid with semilattice of idempotents $E$ and group of units $G$ such that the natural connection $\theta$ is a dual isomorphism from $E$ to a sublattice of $L(G)$, For every  $e_1,\,  e,\, f_1, \, f \in E$, the following statements are true:
\begin{itemize}
\item[${ (i)}$] If $e$ is central in $S$ and  $e\vee f$ exists, then $F_{f,e\vee f}\cong F_{ef,f}$;
\item[${ (ii)}$]  If $e_1$ is central in $e^\upharpoonright$ and  $f_1$ is central in $f^\upharpoonright$, then
$e_1(e \vee f_1)$ is central in $[e_1(e \vee f)]^\upharpoonright$, $f_1(e_1 \vee f)$ is central in $[f_1(e \vee f)]^\upharpoonright$ and $F_{e_1(e \vee f),e_1(e \vee f_1)}\cong F_{f_1(e \vee f),f_1(e_1 \vee f)}$.
\end{itemize}
\end{lem}

\begin{proof}
(i)
For every $x\in f^\upharpoonright \rho_e$, there exists $a\in f^\upharpoonright$ such that $(a,x)\in \rho_e$. It follows that $ex=ea$ and $af=f$. Notice that $e$ is central in $S$, we have
$$ xef=exf=eaf=ef.$$
This shows that $x\geq ef$ and so $f^\upharpoonright \rho_e\subseteq (ef)^\upharpoonright$.

Suppose that $s=hg\in (ef)^\upharpoonright$, where $h\in E$ and $g\in G$. Then we may assume that $h\geq ef$ and $g\geq ef$ by Lemma \ref{fgeq}, that is, $g\in (ef)\theta$. Also, we have that $e\theta\unlhd G$, since $e$ is central in $S$. Notice that $\theta$ is a dual isomorphism from $E$ to a sublattice of $L(G)$ , it follows that
$$ (ef)\theta=(e\wedge f)\theta=(e\theta)\vee(f\theta)=(e\theta)(f\theta)$$
and so there exist $g_1\in e\theta$ and $g_2\in f\theta$ such that $g=g_1g_2$. Let $x=(he\vee f)g_2$. Then we have that $eh\geq ef$ and $eh=eeh\vee ef\geq e(eh\vee f)\geq e(eh)=eh$, since $h\geq ef$. This implies that $e(eh\vee f)=eh$ and
$$ex=e(eh\vee f)g_2=ehg_2=heg_2=heg_1g_2=heg=ehg=es$$
and so $(x,s)\in \rho_e$. Notice that $x\geq f$, we get that $s\in f^\upharpoonright \rho_e$. This shows that $f^\upharpoonright \rho_e\supseteq (ef)^\upharpoonright$ and so $f^\upharpoonright \rho_e= (ef)^\upharpoonright$.

Let $(a,b)\in\rho_{f,e\vee f}$. Then we have that $(e\vee f)a=(e\vee f)b$ and so $ea=e(e\vee f)a=e(e\vee f)b=eb$. It follows that $(a,b)\in \rho_e\cap f^\upharpoonright\times f^\upharpoonright$ and so $\rho_{f,e\vee f}\subseteq \rho_e\cap f^\upharpoonright\times f^\upharpoonright$.
Conversely, suppose that $(s_1,s_2)\in \rho_e$, where $s_i=h_ig_i\in f^\upharpoonright$, $h_i\in E$, $g_i\in G$,  $i=1,2$. Then we have that $es_1=es_2$, $fs_1=f=fs_2$ and we may assume that $g_i\geq f$ and $h_i\geq f$, $i=1,2$, by Lemma \ref{fgeq}. Also, $e\theta\unlhd G$ tells us that $e\theta\cap f\theta\unlhd f\theta$. Notice that $\theta$ is a dual isomorphism from $E$ to a sublattice of $L(G)$ , it follows that
\begin{align*}
[(e\vee f)h_i]\theta&=[(e\vee f)\wedge h_i]\theta=(e\theta\wedge f\theta) \vee h_i\theta  \\
&=(e\theta \cap f\theta)h_i\theta=(e\theta)(h_i\theta)\cap (f\theta)(h_i\theta) \\
&=(e\theta\vee h_i\theta)\wedge (f\theta\vee h_i\theta) =(e\wedge h_i)\theta \wedge (f\wedge h_i)\theta  \\
&=(eh_i \vee fh_i)\theta
\end{align*}
and so $(e\vee f)h_i=eh_i \vee fh_i$, $i=1,2$, since $\theta$ is injective . Notice that $g_ig_i^{-1}=g_i^{-1}g_i=1\geq eh_i, \, fh_i$, $i=1,2$, it follows from \cite[Section 1.4 Proposition 18]{lawson} that
\begin{align*}
(e\vee f)s_1&=(e\vee f)h_1g_1=(eh_1\vee fh_1)g_1=eh_1g_1\vee fh_1g_1 \\
&=es_1\vee fs_1=es_2\vee fs_2=eh_2g_2\vee fh_2g_2=(eh_2\vee fh_2)g_2  \\
&=(e\vee f)h_2g_2=(e\vee f)s_2.
\end{align*}
This shows that $(s_1,s_2)\in \rho_{f,e\vee f}$ and so $\rho_{f,e\vee f}= \rho_e\cap f^\upharpoonright\times f^\upharpoonright$.

It is clear that $\rho_e\cap (ef)^\upharpoonright\times (ef)^\upharpoonright=\rho_{ef,e}$ and so we have that $f^\upharpoonright/\rho_{f,e\vee f}\cong (ef)^\upharpoonright/\rho_{ef,e}$ by Lemma \ref{iso2}. Now, we can see from Lemma \ref{phie} that $F_{f,e\vee f}\cong F_{ef,f}$, as required.

(ii)
Since $f_1$ is central in $f^\upharpoonright$, it follows from Lemma \ref{thecon} (iv) that $e\vee f_1$ is central in $(e\vee f)^\upharpoonright$. Notice that $(e\vee f)^\upharpoonright$ is an inverse subsemigroup of $ e^\upharpoonright$ and $e_1$ is central in $e^\upharpoonright$, we can see from Lemma \ref{thecon} (v) that $e_1(e \vee f_1)$ is central in $[e_1(e \vee f)]^\upharpoonright$. Similarly, We can see that $f_1(e_1 \vee f)$ is central in $[f_1(e \vee f)]^\upharpoonright$.

In the following we prove that $F_{e_1(e \vee f),e_1(e \vee f_1)}\cong F_{f_1(e \vee f),f_1(e_1 \vee f)}$. It is easy to see that $e\vee f_1\geq e\vee f$, since $f_1\geq f$. Thus we can see from (i) that
$$F_{e_1(e \vee f),e_1(e \vee f_1)}=F_{e_1(e \vee f_1)(e \vee f),e_1(e \vee f_1)}
\cong F_{e \vee f,e_1(e \vee f_1)\vee (e \vee f)}.$$
Notice that $e_1\geq e$, $f_1\geq f$, $e_1$ is central in $e^\upharpoonright$ and $\theta$ is a dual isomorphism from $E$ to a sublattice of $L(G)$, we have that $e_1\theta\unlhd e\theta$ and $(e\vee f_1)\theta\leq_g (e\vee f)\theta \leq_g e\theta$. It follows that
\begin{align*}
[e_1(e \vee f_1)\vee (e \vee f)]\theta&=(e_1\theta\vee (e\vee f_1)\theta)\wedge(e\vee f)\theta
=(e_1\theta)((e\vee f_1)\theta)\cap (e\vee f)\theta  \\
&=(e_1\theta \cap (e\vee f)\theta)((e\vee f_1)\theta)=(e_1\theta \cap e\theta\cap f\theta))((e\vee f_1)\theta)  \\
&=(e_1\theta \cap f\theta))((e\vee f_1)\theta) = ((e_1\vee f)\theta)\vee((e\vee f_1)\theta) \\
&=[(e_1\vee f)\wedge (e\vee f_1)]\theta.
\end{align*}
Since $\theta$ is a bijection, this shows that $e_1(e \vee f_1)\vee (e \vee f)=(e_1\vee f)\wedge (e\vee f_1)$ and so
$$ F_{e_1(e \vee f),e_1(e \vee f_1)} \cong F_{e \vee f,e_1(e \vee f_1)\vee (e \vee f)}
=F_{e \vee f,(e_1\vee f)\wedge (e\vee f_1)}.$$
Similarly, we may get that
$F_{f_1(e \vee f),f_1(e_1 \vee f)}\cong F_{e \vee f,(e_1\vee f)\wedge (e\vee f_1)}$
and so
$$F_{e_1(e \vee f),e_1(e \vee f_1)}\cong F_{f_1(e \vee f),f_1(e_1 \vee f)},$$
as required.
\end{proof}

The key to provide an analogue in inverse semigroup theory of Jordan-H\"{o}lder Theorem in group theory is the following proposition.

\begin{prop}\label{schre}
Let $S$ be a factorizable inverse  monoid with semilattice of idempotents $E$ and group of units $G$ such that the natural connection $\theta$ is a dual isomorphism from $E$ to a sublattice of $L(G)$. Then any two subcentral (resp., central) idempotent series of $S$ possess isomorphic subcentral (resp., central) refinements.
\end{prop}

\begin{proof}
Suppose that $0=e_0\leq e_1\leq\cdots\leq e_m=1$ and $0=f_0\leq f_1\leq\cdots\leq f_n=1$ are two subcentral idempotent series of $S$. Let $e_{ij}=e_i(e_{i-1}\vee f_j)$, where $1\leq i\leq m$ and $0\leq j\leq n$ and $f_{ij}=f_j(e_i\vee f_{j-1})$, where $0\leq i\leq m$ and $1\leq j\leq n$. Then we have
$$ e_{i0}=e_i(e_{i-1}\vee f_0)=e_ie_{i-1}=e_{i-1}, \,\,\,\, e_{in}=e_i(e_{i-1}\vee f_n)=e_i1=e_{i},$$
$$ f_{0j}=f_j(e_{0}\vee f_{j-1})=f_jf_{j-1}=f_{i-1}, \,\,\,\, f_{mj}=f_j(e_{m}\vee f_{j-1})=f_j1=f_{j}.$$
Also, it is easy to see that  $e_{i,j-1}\leq e_{ij}$ and $e_{ij}$ is central in $e_{i,j-1}^\upharpoonright$. Similarly, $f_{i-1,j}\leq f_{ij}$ and $f_{ij}$ is central in $f_{i-1,j}^\upharpoonright$. Hence
$$ 0=e_{10}\leq e_{11}\leq \cdots\leq e_{1n}\leq e_{20}\leq\cdots\leq e_{2n}\leq \cdots\leq e_{m0}\leq\cdots\leq e_{mn}=1, \eqno(*)$$
and
$$ 0=f_{01}\leq f_{11}\leq \cdots\leq f_{m1}\leq f_{02}\leq\cdots\leq f_{m2}\leq \cdots\leq f_{0n}\leq\cdots\leq f_{mn}=1 \eqno(**)$$
are refined subcentral idempotent series of $0=e_0\leq e_1\leq\cdots\leq e_m=1$ and $0=f_0\leq f_1\leq\cdots\leq f_n=1$, respectively. We can see from Lemma \ref{iso2pe} (ii) that
$$F_{e_{i,j-1},e_{ij}}=F_{e_i(e_{i-1}\vee f_{j-1}),e_i(e_{i-1}\vee f_j)}
\cong F_{f_j(f_{i-1}\vee e_{i-1}),f_j(f_{j-1}\vee e_i)}=F_{f_{i,j-1},f_{ij}}.$$
Also, it is clear that the lengths of the series $(*)$ and $(**)$ are both $mn$. This shows that $(*)$ and $(**)$ are isomorphic.

In a similar manner we may prove that any two central idempotent series of $S$ possess isomorphic central refinements.
\end{proof}

Notice that a composition subcentral idempotent series can not be refined except by repeating terms. Thus from Proposition \ref{schre} we have immediately the following result, which is an appropriate analogue in inverse semigroup theory of Jordan-H\"{o}lder Theorem in group theory.

\begin{thm}\label{jorh}
Let $S$ be a factorizable inverse  monoid with semilattice of idempotents $E$ and group of units $G$ such that the natural connection $\theta$ is a dual isomorphism from $E$ to a sublattice of $L(G)$. Then any two composition subcentral (resp., central) idempotent series of $S$ are isomorphic.
\end{thm}

Notice that not every inverse semigroup contains a composition subcentral (central) idempotent series.
Recall that a partially ordered set $X$ satisfies the maximal (resp., minimal) condition if every non-empty subset of $X$ contains at least one maximal (resp., minimal) element. Also, $X$ satisfies the ascending (resp., descending) chain condition if there does not exist an infinite properly ascending chain $\lambda_1<\lambda_2<\cdots$ (resp., descending chain $\lambda_1>\lambda_2>\cdots$) in $X$. It is easy to se that maximal (resp., minimal) condition and ascending (respectively descending) chain condition are identical (see \cite[p.66]{robinson}).

Let $S$ be an inverse semigroup, $E_c$ and $E_{sc}$ the set of all central and subcentral idempotents in $S$, respectively. It is easy to see that both $E_c$ and $E_{sc}$ are partially ordered sets under the natural order.
We have

\begin{prop}\label{jorhex}
Let $S$ be a factorizable inverse monoid with semilattice of idempotents $E$ and group of units $G$ such that the natural connection $\theta$ is a dual isomorphism from $E$ to a sublattice of $L(G)$. Then $S$ has a composition subcentral (resp., central) idempotent series if and only if $E_{sc}$ (resp., $E_c$) satisfies ascending and descending chain conditions.
\end{prop}

\begin{proof}
Suppose that $0=e_0< e_1<\cdots< e_m=1$ is a composition subcentral idempotent series of $S$ and $f_1< f_2<\cdots$ is a infinite properly ascending chain of subcentral idempotents. Then we can see from Corollary \ref{subcs} that the chain
$$ 0\leq f_1<\cdots<f_{r+2} $$
of subcentral idempotents  can be refined to a subcentral idempotent series of $S$, and the length of the resulting series is at least $m+1$. However, Theorem \ref{jorh} tells us that its length cannot exceed $m$, a contradiction. This shows that $E_{sc}$ satisfies ascending chain condition. In a similar manner we may prove that $E_{c}$ satisfies descending chain condition.

Conversely, suppose that $E_{sc}$ satisfies ascending and descending chain conditions. If $E_{sc}=\{0,1\}$, then it is easy to see that $0<1$ is a composition subcentral idempotent series of $S$. Suppose that $E_{sc}\neq\{0,1\}$. Then there exists a minimal element $e_1$ in $E_{sc}\setminus \{0,1\}$, since descending chain condition and minimal condition are identical.  Also, there exist minimal elements $e_{i+1}$ in $E_i=(E_{sc}\setminus \{0,1,e_1,\ldots,e_{i}\})\cap e_{i}^\upharpoonright$ for $i\geq 1$ if $E_i\neq\emptyset$. We can therefore get an ascending chain $ e_1<e_2<\cdots$
of subcentral idempotents. Notice that $E_{sc}$ satisfies ascending chain condition, this chain will terminate after finite steps.
That is to say, there exists an integer $n$ such that $E_n=\emptyset$. Let $e_{n+1}=1$. Then we get an ascending chain
$$0=e_0< e_1<e_2<\cdots<e_{n}<e_{n+1}=1 \eqno(*)$$
of subcentral idempotents such that there is no subcentral idempotent $f$ such that $e_i<f<e_{i+1}$ for some $0\leq i\leq n$. We can see from Corollary \ref{subcs} that $(*)$ is a composition subcentral idempotent series of $S$.

The proof that $S$ has a composition central idempotent series if and only if $E_c$ satisfies ascending and descending chain conditions is similar.
\end{proof}

\section{$G$-Nilpotent and $G$-solvable inverse semigroups}

Recall that a normal series
$$1=G_0\unlhd G_1\unlhd \cdots\unlhd G_n=G$$
of a group $G$ is said to be a \emph{central series} (resp., \emph{abelian series}) if $G_{i+1}/G_i \leq_g Z(G/G_i)$ (resp., $G_{i+1}/G_i$ is an abelian group) for all $0\leq i\leq n-1$, and a group $G$ is said to be \emph{nilpotent} (resp., \emph{solvable}) if it has a central (resp., \emph{abelian}) series.
Let $S$ be a factorizable inverse monoid such that the natural connection $\theta$ is a dual isomorphism from $E$ to a sublattice of $L(G)$, where $E$ and $G$ denote the semilattice of idempotents of $S$ and the group of units of $S$ respectively. We shall say that a chain
$$0=e_0\leq e_1\leq\cdots\leq e_n=1$$
of central idempotents in $S$ is a \emph{$G$-nilpotent series} (resp., \emph{$G$-solvable series}) if $F_{e_i,e_{i+1}}\subseteq Z(F_{0,e_{i+1}})$ (resp., $F_{e_i,e_{i+1}}$ is a commutative semigroup) for all $0\leq i\leq n-1$.
An inverse monoid $S$ is said to be \emph{$G$-nilpotent} (resp., \emph{$G$-solvable}) if it has a $G$-nilpotent (resp., \emph{$G$-solvable}) series, and the length of a shortest $G$-nilpotent (resp., \emph{$G$-solvable}) series of $S$ is called the \emph{$G$-nilpotent length} (resp., \emph{$G$-solvable length}) of $S$.
It is easy to see that $G$-nilpotent inverse monoids are $G$-solvable.

\begin{lem}\label{nseq}
Let $S$ be a factorizable inverse monoid such that the natural connection $\theta$ is a dual isomorphism from $E$ to a sublattice of $L(G)$, and 
$$\Gamma: \, 0=e_0\leq e_1\leq\ldots\leq e_n=1$$
a chain of central idempotents in $S$. Then $\Gamma$ is a $G$-nilpotent (resp., $G$-solvable) series if and only if $UF_{e_i,e_{i+1}}$ is contained in $Z(S)$ (resp., is an abelian group) for all $0\leq i\leq n-1$.
\end{lem}

\begin{proof}
It is easy to see that $UF_{e_i,e_{i+1}}$ is an abelian group for all $0\leq i\leq n-1$ if $\Gamma$ is a $G$-solvable series of $S$.
Suppose that $\Gamma$ is a $G$-nilpotent series of $S$, and $x\in UF_{e_i,e_{i+1}}$ for some $0\leq i\leq n-1$. Then we may write $x=e_{i+1}g$ for some $g\in e_i\theta$ by Lemma \ref{fgeq}. Also, we have that $e_{i+1}s=e_{i+1}eh\in F_{0,e_{i+1}}$ for every $s=eh\in S$, where $e\in E$ and $h\in G$. Notice that $e_{i+1}\in Z(S)$ and $e_{i+1}g\in Z(F_{0,e_{i+1}})$, it follows that
$$xs=e_{i+1}geh=e_{i+1}e_{i+1}geh=e_{i+1}g\cdot e_{i+1}eh=e_{i+1}eh\cdot e_{i+1}g= ehe_{i+1}g=sx$$
and so $x\in Z(S)$. This implies that $UF_{e_i,e_{i+1}}\subseteq Z(S)$ for all $0\leq i\leq n-1$.

To show the converse half, suppose that $UF_{e_i,e_{i+1}}\subseteq Z(S)$ for all $0\leq i\leq n-1$,
and $s\in F_{e_i,e_{i+1}}$ for some $0\leq i\leq n-1$. Then we may write $s=fg$ by Lemma \ref{fgeq}, where $g\in e_i\theta$ and $f$ is an idempotent such that $e_i\leq f \leq e_{i+1}$.
Also, we have that $e_{i+1} xy=ye_{i+1} x=e_{i+1} yx$ for every $x\in e_i\theta$ and $y\in G$, since $UF_{e_i,e_{i+1}}\subseteq Z(S)$ and $e_{i+1}$ is central in $S$. It follows that
$e_{i+1}xyx^{-1}y^{-1}=e_{i+1}$, i.e., $xyx^{-1}y^{-1}\in e_{i+1}\theta$ and so
$$ (e_{i+1}\theta x)(e_{i+1}\theta y)=e_{i+1}\theta xy=e_{i+1}\theta yx=(e_{i+1}\theta y)(e_{i+1}\theta x).$$
This shows that $e_i\theta/e_{i+1}\theta\leq_g Z(G/e_{i+1}\theta)$. Notice that $e_{i+1}\theta\leq_g f\theta \leq_g e_i\theta$, we get that $f\theta\unlhd G$ and so $f$ is central in $S$ by Lemma \ref{thecon} (ii). Now, we have $s=fg=fe_{i+1}g\in Z(S)$, since $f, e_{i+1}g\in Z(S)$. This shows that $F_{e_i,e_{i+1}}\subseteq Z(S)$ for all $0\leq i\leq n-1$, i.e., $\Gamma$ is a $G$-nilpotent series of $S$.

Suppose that $UF_{e_i,e_{i+1}}$ is an abelian group for all $0\leq i\leq n-1$, and $s, \, t\in F_{e_i,e_{i+1}}$ for some $0\leq i\leq n-1$. Then we may write $s=fg$ and $t=eh$ by Lemma \ref{fgeq}, where $g, h\in e_i\theta$ and $e, f$ are idempotents such that $e_i\leq f,e \leq e_{i+1}$.
Since $e_{i+1}$ is a central idempotent and $UF_{e_i,e_{i+1}}$ is an abelian group, we have that $e_{i+1} xy=(e_{i+1}x)(e_{i+1}y)=(e_{i+1}y)(e_{i+1}x)=e_{i+1} yx$ for every $x,y\in e_i\theta$. It follows that
$e_{i+1}xyx^{-1}y^{-1}=e_{i+1}$, i.e., $xyx^{-1}y^{-1}\in e_{i+1}\theta$ and so
$$ (e_{i+1}\theta x)(e_{i+1}\theta y)=e_{i+1}\theta xy=e_{i+1}\theta yx=(e_{i+1}\theta y)(e_{i+1}\theta x).$$
This shows that $e_i\theta/e_{i+1}\theta$ is an abelian group. Notice that $e_{i+1}\theta\leq_g e\theta, f\theta\leq_g e_i\theta$, we get that $e\theta,f\theta \unlhd e_{i}\theta$ and so $g(e\theta)g^{-1}=e\theta$, $h^{-1}(f\theta)h=f\theta$. Thus we can see from Lemma \ref{thecon} that $geg^{-1}=e$ and $h^{-1}fh=f$, since $\theta$ is injective. Also, we have that $fe\leq e_{i+1}$, since $f,e\leq e_{i+1}$. It follows that $ghg^{-1}g^{-1}\in e_{i+1}\theta\leq_g (fe)\theta$ and so $fe ghg^{-1}h^{-1} =fe$, i.e., $fegh=fehg$. Now, we have that
$$ st=fgeh=fgeg^{-1}gh=fegh=fehg=efhg=ehh^{-1}fhg=ehfg=ts.$$
This shows that $F_{e_i,e_{i+1}}$ is commutative for all $0\leq i\leq n-1$, i.e. $\Gamma$ is a $G$-solvable series of $S$.
\end{proof}

As two consequences we have immediately

\begin{cor}\label{sei}
Let $S$ be a $G$-nilpotent factorizable inverse monoid and
$$0=e_0\leq e_1\leq\ldots\leq e_n=1$$
a $G$-nilpotent series of $S$. Then all the idempotents in the interval $[e_i,e_{i+1}]$ are central in $S$ for every $0\leq i\leq n-1$. Moreover, if $e\in  [e_i,e_{i+1}]$, then
$$0=e_0\leq e_1\leq\ldots\leq e_i\leq e\leq e_{i+1}\leq\cdots\leq e_n=1$$
is also  a $G$-nilpotent series of $S$.
\end{cor}

\begin{cor}\label{nsl1eq}
Let $S$ be a factorizable inverse monoid such that the natural connection $\theta$ is a dual isomorphism from $E$ to a sublattice of $L(G)$. Then the followings are equivalent.
\begin{itemize}
\item[${ (i)}$] $S$ is a $G$-nilpotent inverse monoid of $G$-nilpotent length 1;

\item[${ (ii)}$] $S$ is a $G$-solvable inverse monoid of $G$-solvable length 1;

\item[${ (iii)}$] $S$ is commutative.
\end{itemize}
\end{cor}

In the following, we shall show that each subgroup of a $G$-nilpotent (resp., $G$-solvable) inverse monoid $S$ of $G$-nilpotent length $n$ (resp., $G$-solvable length $n$) is a nilpotent group of class at most $n$ (resp., a solvable group of derived length at most $n$). For this, we need the following two lemmas.

\begin{lem}\label{gnilu}
Let $S$ be a $G$-nilpotent (resp., $G$-solvable) inverse monoid  with the group $G$ of units of $G$-nilpotent length $n$ (resp., $G$-solvable length $n$). Then $G$ is a nilpotent group (resp., solvable group) of class at most $n$ (resp., of derived length at most $n$).
\end{lem}

\begin{proof}
Suppose that $S$ is $G$-nilpotent. Then there exist central idempotents $e_0,e_1,\ldots,e_n$ in $S$ such that $0=e_0\leq e_1\leq\cdots\leq e_n=1$ is a $G$-nilpotent series of $S$ and so $\Gamma:$
$$ 1=e_n\theta\leq_g e_{n-1}\theta\leq_g\cdots\leq_g e_0\theta=G $$
is a normal series of $G$. It follows that $e_{i+1}xg=ge_{i+1}x=e_{i+1}gx$ for every $x\in e_i\theta$ and $g\in G$, since we have $e_{i+1}x,e_{i+1}\in Z(S)$ by Lemma \ref{nseq}. This implies that $e_{i+1}xgx^{-1}g^{-1}=e_{i+1}$ and so $xgx^{-1}g^{-1}\geq e_{i+1}$, i.e., $xgx^{-1}g^{-1}\in e_{i+1}\theta$.  Thus we have that
$(e_{i+1}\theta x)(e_{i+1}\theta g)=e_{i+1}\theta xg=e_{i+1}\theta gx=(e_{i+1}\theta g)(e_{i+1}\theta x)$
for every $x\in e_i\theta$ and $g\in G$ and so $e_i\theta/e_{i+1}\theta\leq_g Z(G/e_{i+1}\theta)$. This shows that $\Gamma$ is a central idempotent series of $G$ and so
$G$ is a nilpotent group and its nilpotent class is at most $n$, as required.

Suppose that $S$ is $G$-solvable. Then there exist central idempotents $e_0,e_1,\ldots,e_n$ in $S$ such that $0=e_0\leq e_1\leq\cdots\leq e_n=1$ is a $G$-solvable series of $S$ and so $\Gamma:$
$$ 1=e_n\theta\leq_g e_{n-1}\theta\leq_g\cdots\leq_g e_0\theta=G $$
is a normal series of $G$. It follows that
$e_{i+1}xy=(e_{i+1}x)(e_{i+1}y)=(e_{i+1}y)(e_{i+1}x)=e_{i+1}yx$
for every $x,y\in e_i\theta$, since $e_{i+1}\in Z(S)$ and $UF_{e_i,e_{i+1}}$ is abelian. This implies that $e_{i+1}xgx^{-1}g^{-1}=e_{i+1}$ and so $xgx^{-1}g^{-1}\geq e_{i+1}$, i.e., $xgx^{-1}g^{-1}\in e_{i+1}\theta$. Thus we have that
$(e_{i+1}\theta x)(e_{i+1}\theta x)=e_{i+1}\theta xy=e_{i+1}\theta yx=(e_{i+1}\theta y)(e_{i+1}\theta x)$
for every $x,y\in e_i\theta$ and so $e_i\theta/e_{i+1}\theta$ is abelian. This shows that $\Gamma$ is an abelian series of $G$ and so
$G$ is a solvable group and its derived length is at most $n$, as required.
\end{proof}

\begin{lem}\label{fasec}
Let $S$ be a factorizable inverse monoid with group of units $G$ and $e$ an idempotent of $S$. Then $N_e =\{g\in G \mid g^{-1}eg=e\}$ is a subgroup of $G$ and 
the ${\cal H}$-class $H_e$ containing $e$ is isomorphic to factor group $N_e/(e\theta\cap N_e)$.
\end{lem}

\begin{proof}
For every $g_1,g_2\in N_e$, we have that $e=g_1^{-1}eg_1=g_2^{-1}eg_2$. It follows that $g_1eg_1^{-1}=e$ and $g_2^{-1}g_1^{-1}eg_1g_2=g_2^{-1}eg_2=e$ and so $g_1^{-1}, g_1g_2\in N_e$. This shows that $N_e$ is a subgroup of $G$.

It is easy to see that $ge\in H_e$ for every $g\in Ne$, since $(ge)^{-1}(ge)=eg^{-1}ge=e$ and $(ge)(ge)^{-1}=geeg^{-1}=e$ for every $g\in N_e$. Also, if $x=gf\in H_e$, where $g\in G$ and $f\in E$, then $e=e^{-1}e=x^{-1}x=fg^{-1}gf=f$ and $e=ee^{-1}=xx^{-1}=gffg^{-1}=gfg^{-1}=geg^{-1}$. It follows that $x=ge$ and $g^{-1}eg=e$, i.e., $g\in N_e$.
We can therefore define a surjective map $\varphi: N_e\rightarrow H_e$ by $g\varphi=ge$. Notice that $e$ is the identity of $H_e$, we have that
$$(gh)\varphi=ghe=g(he)e=(ge)(he)=(g\varphi)(h\varphi)$$
for every $g,h\in N_e$.
This shows that $\varphi$ is a homomorphism. Now, we only need to show that $\ker\varphi=e\theta\cap N_e$. In fact, for each $x\in N_e$, we have that
$$x\in \ker\varphi \Longleftrightarrow x\varphi=e \Longleftrightarrow xe=e \Longleftrightarrow x\geq e\Longleftrightarrow x\in e\theta.$$
This implies that $\ker\varphi=e\theta\cap N_e$, as required.
\end{proof}

By Lemma \ref{gnilu} and Lemma \ref{fasec} we have immediately

\begin{prop}
Let $S$ be a $G$-nilpotent (resp., $G$-solvable) factorizable inverse monoid of $G$-nilpotent length (resp., $G$-solvable length) $n$. Then each subgroup of $S$ is nilpotent and of class at most $n$ (resp., solvable and of derived length at most $n$).
\end{prop}

The following will give a characterization of $G$-nilpotent (resp., $G$-solvable) inverse monoids.

\begin{thm}\label{niliff}
Let $S$ be a factorizable inverse monoid with semilattice $E$ of idempotents and group $G$ of units such that the natural connection $\theta$ is a dual isomorphism from $E$ to a sublattice of $L(G)$. Then $S$ is $G$-nilpotent (resp., $G$-solvable) if and only if $G$ is nilpotent (resp., solvable) and $E\theta$ contains a central series (resp., an abelian series) of $G$.
\end{thm}

\begin{proof}
We can see from Lemma \ref{gnilu} that $G$ is nilpotent (resp., solvable) and $E\theta$ contains a central series (resp., an abelian series) of $G$ if $S$ is $G$-nilpotent (resp., $G$-solvable).

Suppose that $G$ is nilpotent and $E\theta$ contains a central series
$$1=e_n\theta\leq_g e_{n-1}\theta\leq_g\cdots\leq_g e_0\theta=G$$
of $G$. We can see from Lemma \ref{thecon} (ii) that $e_i$ is a central idempotent for every $0\leq i\leq n$. Also, it is clear that $0=e_0\leq e_1\leq\cdots\leq e_n=1$. In the following we shall prove that
$UF_{e_i,e_{i+1}} \subseteq Z(S)$ for every $0\leq i\leq n$ and so the above central idempotents series of $S$ is a $G$-nilpotent series.
In fact, for $a\in UF_{e_i,e_{i+1}} $ and $s=eh\in S$, where $e\in E$ and $h\in G$, we may write $a=e_{i+1}g$ by Lemma \ref{fgeq}, where $g\in e_i\theta$. Notice that $e_i\theta/e_{i+1}\theta\leq_g Z(G/e_{i+1}\theta)$, we have that
$$e_{i+1}\theta(xy)=(e_{i+1}\theta x)(e_{i+1}\theta y)=(e_{i+1}\theta y)(e_{i+1}\theta x)=e_{i+1}\theta(yx)$$
for every $x\in G$ and $y\in e_i\theta$. In particular, we have that $e_{i+1}\theta(hg)=e_{i+1}\theta(gh)$ and so $hgh^{-1}g^{-1}\in e_{i+1}\theta$. This implies that $e_{i+1} hgh^{-1}g^{-1} =e_{i+1}$, i.e.,
$$e_{i+1}hg=e_{i+1}gh.$$
Also, we have that $x^{-1}g^{-1}xg\in e_{i+1}\theta$
and so $g^{-1}xg=x(x^{-1}g^{-1}xg)\in (e\theta)(e_{i+1}\theta)=(ee_{i+1})\theta$ for every $x\in e\theta$. This implies that $g^{-1}(e\theta)g\leq_g (ee_{i+1})\theta$ and so
\begin{align*}
g^{-1}(ee_{i+1}\theta)g&=g^{-1}(e\theta)(e_{i+1}\theta)g=g^{-1}(e\theta)gg^{-1}(e_{i+1}\theta)g  \\
&=g^{-1}(e\theta)g(e_{i+1}\theta) \leq_g (ee_{i+1})\theta(e_{i+1}\theta)=(ee_{i+1})\theta,
\end{align*}
since $e_{i+1}\theta\unlhd G$.
Similarly we may get that $g(ee_{i+1}\theta)g^{-1}\leq_g (ee_{i+1})\theta$
and so $g^{-1}(ee_{i+1}\theta)g= (ee_{i+1})\theta$. Notice that $\theta$ is injective, it follows from Lemma \ref{thecon} that $g^{-1}(ee_{i+1})g=ee_{i+1}$, i.e.,
$$gee_{i+1}=ee_{i+1}g.$$
Now, we have that
$$as=e_{i+1}geh=gee_{i+1}h=ee_{i+1}gh =ee_{i+1}hg=ehe_{i+1}g=sa. $$
This shows that $UF_{e_i,e_{i+1}} \subseteq Z(S)$, as required.

Suppose that $G$ is solvable and $E\theta$ contains an abelian series
$$1=e_n\theta\leq_g e_{n-1}\theta\leq_g\cdots\leq_g e_0\theta=G$$
of $G$. Similarly, $0=e_0\leq e_1\leq\cdots\leq e_n=1$ is a central idempotent series of $S$ and we only need to prove that $UF_{e_i,e_{i+1}}$ is an abelian group for every $0\leq i\leq n$.
For $s,a\in UF_{e_i,e_{i+1}}$, we may write $s=e_{i+1}h$ and $a=e_{i+1}g$ by Lemma \ref{fgeq}, where $h,g\geq e_i$. Notice that $e_i\theta/e_{i+1}\theta$ is an abelian group, we have that $e_{i+1}\theta(gh)=(e_{i+1}\theta g)(e_{i+1}\theta h)=(e_{i+1}\theta h)(e_{i+1}\theta g)=e_{i+1}\theta (hg)$ and so $ghg^{-1}h^{-1}\in e_{i+1}\theta$. This implies that $e_{i+1}ghg^{-1}h^{-1}=e_{i+1}$, i.e., $e_{i+1}gh=e_{i+1}hg$. Now, we have that
$$ as=e_{i+1}ge_{i+1}h=e_{i+1}gh=e_{i+1}hg=e_{i+1}he_{i+1}g=sa.$$
This shows that $UF_{e_i,e_{i+1}}$ is abelian, as required.
\end{proof}

As a consequence, noticing  that in a coset monoid ${\cal K}(G)$ of a group $G$, the natural connection $\theta$ is a dual isomorphism from $E({\cal K}(G))$ to $L(G)$ (see \cite{mcali1}), we have immediately

\begin{cor}\label{snchr}
A group is a nilpotent (resp., solvable) group of class (derived length) $n$ if and only if its coset semigroup is a $G$-nilpotent (resp., $G$-solvable) monoid of $G$-nilpotent length (resp., $G$-solvable length) $n$.
\end{cor}
The above  provides a characterization of the coset monoids of nilpotent (resp., solvable) groups.

Recall that in \cite{zhao}, an element $a$ in an inverse semigroup $S$ is called \emph{pre-idempotent} if $a^2\leq a$ under the natural order on $S$.
An ${\cal R}$-class $R_e$ containing an idempotent $e$ of an inverse semigroup is said to be \emph{anti-abnormal} if $R_e$ contains exactly one pre-idempotent. An idempotent $e$ in $S$ is \emph{anti-abnormal} if so does $R_e$. In fact, we can show that subcentral idempotentd in an inverse semigroup are anti-abnormal.

\begin{lem}\label{aabtrans}
Let $e,\, f$ and $g$ be idempotents in an inverse semigroup $S$ satisfying that $e\geq f\geq g$. If $e$ is anti-abnormal in $f^\upharpoonright$ and $f$ is anti-abnormal in $g^\upharpoonright$, then $e$ is anti-abnormal in $g^\upharpoonright$.
\end{lem}

\begin{proof}
Suppose that $a\in R_e^{g^\upharpoonright}$ is pre-idempotent. Then we have that $aa^{-1}=e$ and $a^2\leq a$. Let $s=fa$. It is easy to see that $s\in g^\upharpoonright$ and
$$ s^2=fafa\leq faea=faaa^{-1}a=fa^2\leq fa=s,$$
$$ ss^{-1}=faa^{-1}f^{-1}=fef=f.$$
It follows that $s=fa$ is pre-idempotent and so
$fa=f$, since $f$ is anti-abnormal in $g^\upharpoonright$. This implies that $a\geq f$ and so $a\in f^\upharpoonright$. Notice that $e$ is anti-abnormal in $f^\upharpoonright$, we get that $a=e$. This shows that $e$ is anti-abnormal in $g^\upharpoonright$, as required.
\end{proof}

\begin{lem}\label{subcanti}
Every subcentral idempotent in an inverse semigroup is anti-abnormal.
\end{lem}

\begin{proof}
Let $e$ be an central idempotent in an inverse semigroup $S$.
Suppose that $a\in R_e$ is pre-idempotent. Then we have
$aa^{-1}=e$ and $a^2\leq a$. It follows that
$$ a^{-1}a=a^{-1}aa^{-1}a=a^{-1}ea=a^{-1}ae=a^{-1}a^2a^{-1}\leq a^{-1}aa^{-1}\leq a^{-1}$$
and so $a^{-1}a=(a^{-1}a)^{-1}\leq(a^{-1})^{-1}=a$. This implies that $a=aa^{-1}a\leq a^2\leq a$ and so $a^2=a$. Noticing that $a\in R_e$, we get that $a=e$. This shows that $e$ is anti-abnormal. Now, the conclusion follows from Lemma \ref{aabtrans} immediately.
\end{proof}

The following gives some further properties of $G$-nilpotent inverse monoids.

\begin{prop}\label{nilsubc}
Let $S$ be a factorizable inverse monoid with semilattice $E$ of idempotents and group $G$ of units such that the natural connection $\theta$ is a dual isomorphism from $E$ to a sublattice of $L(G)$. If $S$ is $G$-nilpotent, then
\begin{itemize}
\item[${(i)}$] every idempotent $f\in E$ is subcentral in $S$;

\item[${(ii)}$]  there exists an idempotent $e<f$ such that $f$ is central in $e^\upharpoonright$ for every idempotent $f\in E\setminus\{0\}$;

\item[${(iii)}$]  every 0-minimal idempotent is central in $S$;

\item[${(iv)}$]  $S$ is $E$-reflexive.
\end{itemize}
\end{prop}

\begin{proof}
We only need to prove (i), since (ii) and (iii) follows from (i) immediately and (iv) is a immediate consequence of Lemma \ref{subcanti} and \cite[Lemma 3.2]{zhao}.
Suppose that $0=e_0\leq e_1\leq\cdots\leq e_n=1$ is a $G$-nilpotent series of $S$. Let $f_i=e_if$ for each $0\leq i\leq n$. For $x=hg\in f_i^\upharpoonright$, where $h\in E$ and $g\in G$, we can see from Lemma \ref{fgeq} that $g\geq f_i$. Notice that $\theta$ is a dual isomorphism from $E$ to a sublattice of $L(G)$ and $e_i$ is central in $S$, we have that
$$g\in f_i^\upharpoonright \cap G=(e_if)^\upharpoonright \cap G=(e_if)\theta=e_i\theta\vee f\theta=( e_i\theta)(f\theta).$$
This implies that there exist $g_1\in e_i\theta$ and $ g_2\in f\theta$ such that $g=g_1g_2$. It follows that $e_{i+1}g_1\in Z(S)$ and $g_2f=fg_2=f$ and so
\begin{align*}
xf_{i+1}&=hge_{i+1}f= he_{i+1}gf=he_{i+1}g_1g_2f=he_{i+1}g_1f \\
&=hfe_{i+1}g_1=hfg_2e_{i+1}g_1=hfe_{i+1}g_1g_2=e_{i+1}fhg_1g_2 \\
&=f_{i+1}x
\end{align*}
This shows that $f_{i+1}$ is central in $f_i^\upharpoonright$ and so
$$ 0=f_0\leq f_1\leq\cdots\leq f_n=f$$
is a subcentral idempotent series between 0 and $f$. Thus $f$ is subcentral in $S$, as required.
\end{proof}

\section{Some problems}

We know that the following problems in group theory (see \cite{khukhro}) are still open:

\begin{itemize}
\item[$*$] {\bf Problem 1}(\cite[Problem 15.40]{khukhro}) Let $N$ be a nilpotent subgroup of a finite simple group $G$. Is it true that there exists a subgroup $N_1$ conjugate to $N$ such that $N \cap N_1 = 1$?
\item[$*$] {\bf Problem 2}(\cite[Problem 16.97]{khukhro}) Let $G$ be a torsion-free group with all subgroups subnormal of defect at most $n$. Must $G$ then be nilpotent of class at most $n$?
\item[$*$] {\bf Problem 3}(\cite[Problem 17.41]{khukhro}) Let $S$ be a solvable subgroup of a finite group $G$ that has no nontrivial solvable normal subgroups.

a) (L. Babai, A. J. Goodman, L. Pyber). Do there always exist seven conjugates of $S$ whose intersection is trivial?

b) Do there always exist five conjugates of $S$ whose intersection is trivial?
\item[$*$] {\bf Problem 4}(\cite[Problem 17.91]{khukhro}) Let $d(X)$ denote the derived length of a group $X$.

a) Does there exist an absolute constant $k$ such that $d(G)-d(M) \leq k$ for every finite soluble group $G$ and any maximal subgroup $M$?

b) Find the minimum $k$ with this property.
\item[$*$] {\bf Problem 5}(\cite[Problem 18.61]{khukhro}) Is a 2-group nilpotent if all its finite subgroups are nilpotent of class at most 3?
\item[$*$] {\bf Problem 6}(\cite[Problem 20.122]{khukhro}) Let $G$ be a finite group and $F(G)$ denotes the largest nilpotent normal subgroup of $G$. For nilpotent subgroups $A, B, C$ of $G$, let ${\rm Min}_G(A, B, C)$ be the subgroup of $A$ generated by all minimal by inclusion intersections of the form $A\cap B^x\cap C^y$, where $x, y \in G$, and let ${\rm min}_G(A, B, C)$ be the subgroup of ${\rm Min}_G(A, B, C)$ generated by all intersections of this kind of minimal order.

a) Is it true that ${\rm min}_G(A, B, C) \leq F(G)$?

b) Is it true that ${\rm Min}_G(A, B, C) \leq F(G)$?

c) The same questions for soluble groups.
\end{itemize}

In the following we shall give a translation of the above problems to semigroup theory, which may be useful for us to solve these open problems.

It was shown in \cite{mcali1} that an inverse semigroup $S$ is isomorphic to the coset monoid ${\cal K}(G)$ for some group $G$ if and only if $S=EH$ is factorizable and the natural connection $\theta$ of the semilattice $E$ idempotents of $S$ into the lattice $L(H)$ of subgroups of $H$ is a dual isomorphism, where $E$ is the semilattice of idempotents of $S$ and $H$ is the group of units of $S$. In this case, $H\cong G$. Further, the subgroups (resp., normal subgroups) of group $G$ are precisely the idempotents (resp., central idempotents) of the coset monoid ${\cal K}(G)$ of $G$ and two subgroups of $G$ are conjugate if and only if they are ${\cal D}$-equivalent in ${\cal K}(G)$.
Also,  1 (the identity subgroup of $G$) and $G$ are the identity and zero element of ${\cal K}(G)$, respectively. Hence, a group is simple if and only if its coset monoid has no central idempotents except zero and identity.
For a subgroup $H$ of $G$, we can see from \cite{zhao} that the coset monoid ${\cal K}(H)$ of $H$ is equal to the filter $H^\upharpoonright$ generated by $H$ in ${\cal K}(G)$.
This implies that ${\cal K}(H\cap K)={\cal K}(H)\cap {\cal K}(K)=H^\upharpoonright\cap K^\upharpoonright$ for every subgroups $H$ and $K$ of $G$.
Notice that the proof of \cite[Proposition 3.3]{zhao} tells us that in a factorizable inverse monoid, two idempotents $e$ and $f$ are ${\cal D}$-equivalent if and only if there exists a unit element $g$ such that $f=g^{-1}eg$. In this case, it is easy to verify that $f^\upharpoonright=g^{-1}(e^\upharpoonright)g$.

In the remainder of this section, we always assume that $S$ is a factorizable inverse monoid with semilattice $E$ of idempotents and group $G$ of units, and the natural connection $\theta$ is a dual isomorphism from $E$ to the lattice $L(G)$ of subgroups $G$, i.e., $S$ is isomorphic to  the coset monoid ${\cal K}(G)$ of group $G$.
Now, it is easy to see that Problem 1 is equivalent to the following problem in semigroup theory.

\begin{itemize}
\item[$*$] {\bf Problem $\bf 1^\prime$} If $S$ is finite and $E_c=\{0,1\}$, does there exist $g\in G$ such that $e^\upharpoonright \cap g^{-1}(e^\upharpoonright)g = \{1\}$ for any given $e\in E$ with $G$-nilpotent $e^\upharpoonright$?
\end{itemize}

It is easy to see that a subgroup of a group $G$ is a subnormal subgroup of defect $n$ if and only if it is a subcentral idempotent in ${\cal K}(G)$ of defect $n$. We can see from Corollary \ref{snchr} that $G$ is a nilpotent (resp., solvable) group of class (resp., derived length) $n$ if and only if its coset monoid ${\cal K}(G)$ is a $G$-nilpotent (resp., $G$-solvable) monoid of $G$-nilpotent length (resp., $G$-solvable length) $n$.
Thus, Problem 2 and Problem 3 are equivalent to the following problems in semigroup theory respectively.

\begin{itemize}
\item[$*$] {\bf Problem $\bf 2^\prime$} If $G$ is torsion-free and every idempotent in $S$ is subcentral of defect at most $n$, must $S$ then be $G$-nilpotent of $G$-nilpotent length at most $n$?
\item[$*$] {\bf Problem $\bf 3^\prime$} If $S$ is finite and $f^\upharpoonright$ is not $G$-solvable for every central idempotent $f\in E_c\setminus\{1\}$, does there always exist seven (or five) units in $G$, say $g_1,g_2,\ldots,g_7$ (or $g_1,g_2,\ldots,g_5$), such that $\bigcap\{g_i^{-1}(e^\upharpoonright)g_i \mid 1\leq i\leq 7\}=\{1\}$ (or $\bigcap\{g_i^{-1}(e^\upharpoonright)g_i \mid 1\leq i\leq 5\}=\{1\}$) for any given $e\in E$ with $G$-solvable $e^\upharpoonright$?
\end{itemize}

Notice that the maximal subgroups of a group $G$ are precisely the primitive idempotents of ${\cal K}(G)$. Let $d_s(X)$ denote the $G$-solvable length of a $G$-solvable inverse semigroup $X$. Then Problem 4 and  Problem 5 are translated as follows:
\begin{itemize}
\item[$*$] {\bf Problem $\bf 4^\prime$}
a) Does there exist an absolute constant $k$ such that $d_s(S)-d_s(e^\upharpoonright) \leq k$ for every finite $G$-solvable inverse monoid $S$ and any primitive idempotent $e$ in $S$?

b) Find the minimum $k$ with this property.
\item[$*$] {\bf Problem $\bf 5^\prime$} If $G$ is a 2-group and $e^\upharpoonright$ is $G$-nilpotent of $G$-nilpotent length at most 3 for every $e$ in $E$ with $e^\upharpoonright$ finite, must $S$ then be $G$-nilpotent?
\end{itemize}

Notice that the natural partial order on the coset monoid ${\cal K}(G)$ of a group $G$ is the inverse of inclusion, the intersection $H\cap K$ of two subgroups $H$ and $K$ in a group $G$ is precisely the join $H\vee K$ of $H$ and $K$ in ${\cal K}(G)$. Also, a subgroup $H$ of $G$ is contained in $F(G)$ if and only if there exists a nilpotent normal subgroup $K$ such that $H\leq K$, since $F(G)$ is the largest nilpotent normal subgroup of $G$. Thus, Problem 6 is equivalent to the following problem in semigroup theory.

\begin{itemize}
\item[$*$] {\bf Problem $\bf 6^\prime$} Let $S$ be finite, $e,f,h\in E$ such that $e^\upharpoonright.f^\upharpoonright,h^\upharpoonright$ are $G$-nilpotent (resp., $G$-solvable),
and let $M=\{e\vee f'\vee h' \mid f',h'\in E, f'{\cal D} f, h'{\cal D} h\}$,
${\rm Max}_S(e, f, h)$ be the product of all maximal elements (under natural order) in $M$,
$m=min\{|x\theta| \mid x\in M\}$ and ${\rm max}_S(e, f, h)$ be the product all elements $x$ in $M$ with $|x\theta|=m$.

a) Does there exist a central idempotent $c$ in $S$ such that $c^\upharpoonright$ is $G$-nilpotent (resp., $G$-solvable) and $c\leq {\rm max}_S(e, f, h)$?

b) Does there exist a central idempotent $c$ in $S$ such that $c^\upharpoonright$ is $G$-nilpotent (resp., $G$-solvable) and $c\leq {\rm Max}_S(e, f, h)$?
\end{itemize}

\end{document}